\documentclass[10pt, english]{elsarticle}
\usepackage{amsthm}
\usepackage{amsmath}
\usepackage{latexsym, amssymb}
\usepackage{txfonts}
\usepackage{mathtools}
\usepackage{color}
\usepackage{babel}
\usepackage[all]{xy}
\usepackage[capitalise]{cleveref}

\newtheorem{thm}{Theorem}[section] 

\newtheorem{cor}[thm]{Corollary}

\newtheorem{lem}[thm]{Lemma}
\newtheorem{prop}[thm]{Proposition}

\theoremstyle{definition}

\newtheorem{exmpl}[thm]{Example}

\newcommand\operA[2]{{\if!#2!\operatorname{#1}\else{\operatorname{#1}_{#2}^{\phantom{I}}}\fi}} 

\def\tr{{\operatorname{Tr}}}

\def\norm{{\operatorname{N}}}

\newcommand{\Trace}[1][]{\if!#1!\operatorname{Tr}\else{\operatorname{Tr}_{#1}^{\phantom{I}}}\fi} 

\long\def\forget#1\forgotten{{}} %

\def\({\left(}
\def\){\right)}

\newcommand\LAY[3][]{{\begin{array}{c}\mbox{#2} \if#1!{}\else{+}\fi \\ \mbox{#3}\end{array}}}

\makeatletter
\newcommand{\bigperp}{%
  \mathop{\mathpalette\bigp@rp\relax}%
  \displaylimits
}

\newcommand{\bigp@rp}[2]{%
  \vcenter{
    \m@th\hbox{\scalebox{\ifx#1\displaystyle2.1\else1.5\fi}{$#1\perp$}}
  }%
}
\makeatother

\renewcommand{\geq}{\geqslant}
\renewcommand{\leq}{\leqslant}

\makeatletter
\def\ps@pprintTitle{%
 \let\@oddhead\@empty
 \let\@evenhead\@empty
 \def\@oddfoot{\centerline{\thepage}}%
 \let\@evenfoot\@oddfoot}
\makeatother

\newif\iffurther
\furtherfalse

\begin{document}
\begin{frontmatter}

\title{Linkage of sets of Quaternion Algebras in characteristic 2}

\author{Adam Chapman}
\ead{adam1chapman@yahoo.com}
\address{School of Computer Science, Academic College of Tel-Aviv-Yaffo, Rabenu Yeruham St., P.O.B 8401 Yaffo, 6818211, Israel}

\begin{abstract}
This note contains two new observations on the linkage properties of quaternion algebras over fields of characteristic 2: first, that a 3-linked field need not be 4-linked (a case which was left open in previous papers) and that three inseparably linked quaternion algebras are also cyclically linked when the base-field is odd-closed.
 \end{abstract}

\end{frontmatter}

\section{Introduction}

A symbol $p$-algebra of prime degree $p$ over a field $F$ of $\operatorname{char}(F)=p$ is an algebra of the form
$[\alpha,\beta)_{p,F}=F \langle x,y : x^p-x=\alpha,y^p=\beta, y x y^{-1}=x+1 \rangle$
for some $\alpha \in F$ and $\beta \in F^\times$.
The significance of these algebras derives from the fact (proven by Teichm\"uller, see \cite[Section 9]{GS}) that ${_pBr}(F)$ is generated by the Brauer classes of these algebras.
Symbol $p$-algebras of degree 2 over a field $F$ of $\operatorname{char}(F)=2$ are the quaternion algebras over that field.

We say that symbol $p$-algebras $A_1,\dots,A_m$ of degree $p$ over $F$ are linked if there exists a degree $p$ field extension $K/F$ that splits them all. We say they are inseparably linked if there exists a purely inseparable degree $p$ field extension $K=F[\sqrt[p]{\gamma}]$ that splits them all, and we say these algebras are cyclically linked if there exists a degree $p$ cyclic extension $K=F[\lambda : \lambda^p-\lambda=\delta]$ which splits them all.
It was proven in \cite{Chapman:2015} that if $A_1$ and $A_2$ are inseparably linked then they are also cyclically linked and counterexamples were provided for the converse statement.
The special case of quaternion algebras had been proven earlier in \cite{Draxl:1983}, \cite{Lam:2002} and \cite{ElduqueVilla:2005}.
Note that quadratic field extensions are either cyclic or purely inseparable, and therefore if quaternion algebras $A_1,\dots,A_m$ are linked then they are either cyclically or inseparably linked.

We say a field $F$ is $m$-linked if any $m$ quaternion algebras over $F$ are linked. Examples of 2-linked fields that are not 3-linked exist: $\mathbb{F}_2(\!(\alpha)\!)(\!(\beta)\!)$ is 2-linked by \cite{AravireJacob:1995} but not 3-linked by \cite{ChapmanDolphinLeep}.
The question of whether 3-linked fields which are not 4-linked exist was raised in \cite{Becher}. Examples were provided in \cite{ChapmanTignol:2019} for $\operatorname{char}(F) \neq 2$. Similar questions in characteristic 2 were studied in \cite{Chapman:2018}, but this specific problem remained open because the tools used in those previous works proved inadequate.
In this note we complete the picture by proving the existence of 3-linked fields of characteristic 2 which are not 4-linked. The tool is the $w$-invariant for valued division algebras defined and studied by Tignol in \cite{Tignol:1992} and other works.
We also extend the result from \cite{Chapman:2015} from pairs of symbol $p$-algebras to triples, showing that if three symbol $p$-algebras of prime degree $p$ over a $p$-special field $F$ of characteristic $p$ are inseparably linked then they are also cyclically linked.

\section{3-Linked Fields which are not 4-Linked}\label{Examplen2}

Let $F=k(\!(\alpha)\!)(\!(\beta)\!)$ be the field of iterated Laurent series in two variables over a field $k$ of $\operatorname{char}(k)=2$.
Write $\mathfrak{v}$ for the right-to-left $(\alpha,\beta)$-adic valuation on $F$ and write $\Gamma_F=\mathbb{Z} \times \mathbb{Z}$ for the underlying value group. This is a henselian valuation, and therefore it extends uniquely to any field extension $K/F$ by the formula $\mathfrak{v}(t)=\frac{1}{d} \mathfrak{v}(\norm(t))$ where $\norm$ is the norm form $K \rightarrow F$ and $d=[K:F]$. Consequently, the valuation extends to any central division algebra $D/F$ with the analogous formula $\mathfrak{v}(t)=\frac{1}{d} \mathfrak{v}(\norm(t))$ where $\norm$ here is the reduced norm $D \rightarrow F$ and $d$ is the degree of $D$ (see \cite{TignolWadsworth:2015} for background on valuation theory for division algebras).
Following \cite{Tignol:1992}, for any cyclic quadratic field extension $K$ of $F$, we define $w(K)$ to be $\min \{ \mathfrak{v}(\tr(t))-\mathfrak{v}(t) : t \in K\}$ where $\tr : K \rightarrow F$ is the trace map, given by $\tr(ax+b)=a$ where $K=F[\wp^{-1}(\gamma)]=F[x : x^2+x=\gamma]$.
Similarly, for any quaternion division algebra $D$ over $F$, we define $w(D)$ to be
$\min \{ \mathfrak{v}(\tr(t))-\mathfrak{v}(t) : t \in D\}$, 
where $\tr : D \rightarrow F$ is the reduced trace map given by $\tr(a+bx+cy+xy)=b$ where $D=[\gamma,\delta)_{2,F}=F \langle x,y :x^2+x=\gamma, y^2=\delta, y x y^{-1}=x+1 \rangle$.
Note that for any quaternion division algebra $D$ over $F$ containing a quadratic cyclic extension $K$ of $F$, the restriction of the reduced trace map on $D$ to $K$ coincides with the trace map on $K$, and therefore $w(K)\geq w(D)$.

\begin{lem}[{\cite[Theorem 1.8, with $p=2$]{Tignol:1992}}]\label{Tignol}
If $\gamma$ is an element in $F$ with $\mathfrak{v}(\gamma)<\vec{0}$ and $\mathfrak{v}(\gamma) \not \in 2 \Gamma_F$, then $K=F[t : t^2+t=\gamma]$ is a totally ramified quadratic cyclic field extension of $F$ with $w(K)=-\mathfrak{v}(t)=-\frac{1}{2} \mathfrak{v}(\gamma)$.
\end{lem}

\begin{lem}
Consider the quaternion algebras $[\alpha^{-1},\beta)_{2,F}$, $[\beta^{-1},\alpha)_{2,F}$ and $[\alpha^{-1}\beta^{-1},\beta)_{2,F}$ over $F=k(\!(\alpha)\!)(\!(\beta)\!)$ where $k$ is a field of $\operatorname{char}(k)=2$.
If a field $K$ is a quadratic extension of $F$ that splits these algebras, then $w(K) \leq (\frac{1}{2},\frac{1}{2})$.
\end{lem}

\begin{proof}
These algebras do not share any quadratic inseparable splitting field by \cite{Chapman:2018}. 
Note that they satisfy the conditions of \cite[Proposition 2.1]{Tignol:1992}, and therefore they are division algebras.
Suppose $K$ is a common quadratic  splitting field of theirs, and therefore it must be a cyclic extension of $F$.
Since $K$ is a quadratic splitting field of $[\alpha^{-1},\beta)_{2,F}=F \langle x,y :x^2+x=\alpha^{-1}, y^2=\beta, y x y^{-1}=x+1 \rangle$, it is also a subfield of that algebra, and since it is a cyclic extension of $F$, it must be generated over $F$ by an element of trace $1$, i.e. an element $t$ of the form $t=\lambda+x+ay+bxy$ for some $\lambda,a,b\in F$. Since $x+ay+bxy$ generates $K$ too and it is of trace 1, we can suppose $\lambda=0$ and take $t=x+ay+bxy$ to be the generator of $K$.
Now, $\mathfrak{v}(x)=(-\frac{1}{2},0)$, $\mathfrak{v}(y)=(0,\frac{1}{2})$ and $\mathfrak{v}(xy)=(-\frac{1}{2},\frac{1}{2})$. Since they belong to three distinct classes modulo $\Gamma_F$, we have $\mathfrak{v}(t)=\min\{\mathfrak{v}(x),\mathfrak{v}(a y),\mathfrak{v}(bxy)\}$.
Hence, if $\mathfrak{v}(t)<(-\frac{1}{2},-\frac{1}{2})$ then $\mathfrak{v}(t)<(-\frac{1}{2},0)$ and so $\mathfrak{v}(t)\neq \mathfrak{v}(x)$, which means $\mathfrak{v}(t)=\min\{\mathfrak{v}(ay),\mathfrak{v}(bxy)\}$, and therefore the class of $\mathfrak{v}(t)$ modulo $\Gamma_F$ is either $(\frac{1}{2},\frac{1}{2})$ or $(0,\frac{1}{2})$. Since $w(K)=-\mathfrak{v}(t)$ by Lemma \ref{Tignol}, if $w(K)>(\frac{1}{2},\frac{1}{2})$ then the class of $w(K)$ modulo $\Gamma_F$ is either $(\frac{1}{2},\frac{1}{2})$ or $(0,\frac{1}{2})$.
Recalling that $K$ is also a quadratic splitting field of $[\beta^{-1},\alpha)_{2,F}$, the same argument shows that if $w(K)>(\frac{1}{2},\frac{1}{2})$ then the class of $w(K)$ modulo $\Gamma_F$ is either $(\frac{1}{2},\frac{1}{2})$ or $(\frac{1}{2},0)$.
Similarly, since $K$ is a quadratic splitting field of $[\alpha^{-1}\beta^{-1},\beta)_{2,F}$, if $w(K)>(\frac{1}{2},\frac{1}{2})$ then the class of $w(K)$ modulo $\Gamma_F$ is either $(0,\frac{1}{2})$ or $(\frac{1}{2},0)$.
However, the intersection of $\{(\frac{1}{2},\frac{1}{2}),(0,\frac{1}{2})\}$, $\{(\frac{1}{2},\frac{1}{2}),(\frac{1}{2},0)\}$ and $\{(\frac{1}{2},0),(0,\frac{1}{2})\}$ is trivial, and so it is impossible that $w(K)>(\frac{1}{2},\frac{1}{2})$, which means that $w(K) \leq (\frac{1}{2},\frac{1}{2})$.
\end{proof}

\begin{prop}\label{No4}
The algebras $[\alpha^{-1},\beta)_{2,F}$, $[\beta^{-1},\alpha)_{2,F}$, $[\alpha^{-1} \beta^{-1},\beta)_{2,F}$ and $[\alpha^{-2} \beta^{-1},\alpha)_{2,F}$ are not linked.
\end{prop}

\begin{proof}
All the quadratic splitting fields $K$ the first three share satisfy $w(K) \leq (\frac{1}{2},\frac{1}{2})$.
On the other hand, by \cite[Proposition 2.1]{Tignol:1992} $[\alpha^{-2} \beta^{-1},\alpha)_{2,F}$ is a division algebra with $w([\alpha^{-2} \beta^{-1},\alpha)_{2,F})=(1,\frac{1}{2})$, which is greater than $(\frac{1}{2},\frac{1}{2})$, so every quadratic splitting field $K$ of $[\alpha^{-2} \beta^{-1},\alpha)_{2,F}$ has $w(K)\geq w([\alpha^{-2} \beta^{-1},\alpha)_{2,F})>(\frac{1}{2},\frac{1}{2})$. Therefore there is no common quadratic splitting field of the first three algebras which is also a splitting field of the fourth.
\end{proof}

\begin{exmpl}
The field $F=k(\!(\alpha)\!)(\!(\beta)\!)$ is 3-linked when $k$ is algebraically closed (e.g., $k=\mathbb{F}_2^{sep}$).
The reason is that given any three quaternion algebras $[\gamma_1,\delta_1)_{2,F}$, $[\gamma_2,\delta_2)_{2,F}$ and $[\gamma_3,\delta_3)_{2,F}$, the system
\begin{eqnarray*}
b^2 \gamma_1+(c_1^2+c_1 d_1+d_1^2 \gamma_1)\delta_1 & = & a_2^2+a_2b+b^2 \gamma_2+(c_2^2+c_2 d_2+d_2^2 \gamma_2)\delta_2 \\
b^2 \gamma_1+(c_1^2+c_1 d_1+d_1^2 \gamma_1)\delta_1 & = & a_3^2+a_3b+b^2 \gamma_3+(c_3^2+c_3 d_3+d_3^2 \gamma_3)\delta_3
\end{eqnarray*}
is a system of two homogeneous quadratic equations in 9 variables 
$$a_2,a_3,b,c_1,c_2,c_3,d_1,d_2,d_3,$$ \sloppy and since $F$ is a $C_2$-field (see \cite{Lang:1952} and \cite[Chapter 6]{EnglerPrestel:2005}) and $9>2\cdot 2^2$, the system has a nontrivial solution. If this nontrivial solution has $b=0$, then $F[\sqrt{(c_1^2+c_1 d_1+d_1^2 \gamma_1)\delta_1}]$ is a common inseparable splitting field of the three quaternion algebras, and if $b\neq 0$ then $F[t : t^2+t=\gamma_1+(\frac{c_1^2}{b^2}+\frac{c_1 d_1}{b^2}+\frac{d_1^2}{b^2} \gamma_1)\delta_1]$ is a common cyclic splitting field of the three quaternion algebras.
However, $F$ is not 4-linked by Proposition \ref{No4}.
This provides a characteristic 2 analogue of \cite{ChapmanTignol:2019}.
\end{exmpl}

\section{Cyclic and Inesparable 3-Linkage of Symbol $p$-Algebras}\label{TripleLinkage}

In \cite{Chapman:2015} it was proven that if two symbol $p$-algebras of degree $p$ over a field $F$ of $\operatorname{char}(F)=p$ are inseparably linked then they are also cyclically linked.
Here we prove that the same holds for three algebras, under the assumption that $F$ is $p$-special, i.e. has no finite field extensions of degree prime to $p$.

\begin{thm}
Given a prime integer $p$ and a field $F$ of $\operatorname{char}(F)=p$, if three symbol $p$-algebras of degree $p$ over $F$ are inseparably linked then they become cyclically linked under a prime-to-$p$ extension of $F$ of degree at most $2p-1$.
\end{thm}

\begin{proof}
The algebras are inseparably linked, so they can be written as $A_1=[\alpha_1,\beta)_{p,F}$, $A_2=[\alpha_2,\beta)_{p,F}$ and $A_3=[\alpha_3,\beta)_{p,F}$.
We can suppose that $\langle [A_1],[A_2],[A_3] \rangle$ is a subgroup of ${_pBr}(F)$ of order $p^3$, because otherwise it is enough to stress that two of them are cyclically linked (which is true by \cite{Chapman:2015}) and the third shares the same cyclic field extension of $F$ shared by the first two.
This means that $\alpha_1,\alpha_2,\alpha_3 \not \in \wp(F)=\{\lambda^p-\lambda : \lambda \in F\}$ and in particular they are nonzero.
Moreover, $\alpha_2 \not \in \mathbb{F}_p \alpha_1$, and so $(\frac{\alpha_1}{\alpha_2})^{p-1} \neq 1$.

Now, $A_i$ can also be written as $[\alpha_i^p,\beta)_{p,F}$, and a solution to the following system will show the algebras are cyclically linked:
\begin{eqnarray*}
\alpha_1^p+\beta(u_1^p-u_1 v_1^{p-1}+\alpha_1^p v_1^p) & = & \alpha_2^p+\beta(u_2^p-u_2 v_2^{p-1}+\alpha_2^p v_2^p)\\
\alpha_1^p+\beta(u_1^p-u_1 v_1^{p-1}+\alpha_1^p v_1^p) & = & \alpha_3^p+\beta(u_3^p-u_3 v_3^{p-1}+\alpha_3^p v_3^p).
\end{eqnarray*}
Take $u_1=u_2=u_3$, and then the equations reduce to
\begin{eqnarray*}
\alpha_1^p+\beta(-u_1 v_1^{p-1}+\alpha_1^p v_1^p) & = & \alpha_2^p+\beta(-u_2 v_2^{p-1}+\alpha_2^p v_2^p)\\
\alpha_1^p+\beta(-u_1 v_1^{p-1}+\alpha_1^p v_1^p) & = & \alpha_3^p+\beta(-u_3 v_3^{p-1}+\alpha_3^p v_3^p).
\end{eqnarray*}
and then isolate $u_1$ in both equations
\begin{eqnarray*}
\beta(u_1 v_2^{p-1}-u_1 v_1^{p-1}) & = & \alpha_2^p-\alpha_1^p+\beta(\alpha_2^p v_2^p-\alpha_1^p v_1^p)\\
\beta(u_1 v_3^{p-1}-u_1 v_1^{p-1}) & = & \alpha_3^p-\alpha_1^p+\beta(\alpha_3^p v_3^p-\alpha_1^p v_1^p).
\end{eqnarray*}
Take $v_2=\frac{\alpha_1}{\alpha_2} v_1$ and $v_3=1$, and then
\begin{eqnarray*}
\beta\left(\left(\frac{\alpha_1}{\alpha_2}\right)^{p-1}-1\right) u_1 v_1^{p-1} & = & \alpha_2^p-\alpha_1^p\\
\beta(u_1-u_1 v_1^{p-1}) & = & \alpha_3^p-\alpha_1^p+\beta(\alpha_3^p-\alpha_1^p v_1^p).
\end{eqnarray*}
Plugging in the first equation in the second gives
$$\beta u_1-\frac{\alpha_2^p-\alpha_1^p}{(\frac{\alpha_1}{\alpha_2})^{p-1}-1}=\alpha_3^p-\alpha_1^p+\beta(\alpha_3^p-\alpha_1^p v_1^p), \quad \text{and so}$$
$$\beta u_1=\alpha_3^p-\alpha_1^p+\frac{\alpha_2^p-\alpha_1^p}{(\frac{\alpha_1}{\alpha_2})^{p-1}-1}+\beta(\alpha_3^p-\alpha_1^p v_1^p).$$
Plugging this into the first equation gives
$$(\alpha_3^p-\alpha_1^p+\frac{\alpha_2^p-\alpha_1^p}{(\frac{\alpha_1}{\alpha_2})^{p-1}-1}+\beta(\alpha_3^p-\alpha_1^p v_1^p)) \left(\left(\frac{\alpha_1}{\alpha_2}\right)^{p-1}-1 \right) v_1^{p-1}=\alpha_2^p-\alpha_1^p,$$
 a degree $2p-1$ equation in one variable $v_1$, that must have a solution under some prime-to-$p$ extension of $F$ of degree up to $2p-1$.
\end{proof}

\begin{cor}
\
\begin{enumerate}
\item For $p$-special fields, inseparable linkage of three symbol $p$-algebras of degree $p$ implies cyclic linkage.
\item When $p=2$, every three inseparably linked quadratic 2-fold Pfister forms over $F$ are either cyclically linked or become cyclically linked over a degree 3 extension of $F$.
\end{enumerate}
\end{cor}

\sloppy Note that three cyclically linked symbol $p$-algebras need not be inseparably linked. For example, the algebras $[\alpha^{-1},\beta_1)_{p,F}$, $[\alpha^{-1},\beta_2)_{p,F}$ and $[\alpha^{-1},\beta_3)_{p,F}$ over $F=k(\!(\alpha)\!)(\!(\beta_1)\!)(\!(\beta_2)\!)(\!(\beta_3)\!)$ (or over $F_0=k(\alpha,\beta_1,\beta_2,\beta_3)$) 
are pair-wise not inseparably linked by the same argument as in \cite[Example 4.2]{Chapman:2015}.

When $p=2$, these algebras remain pair-wise not inseparably linked under any odd field extension of $F$. The reason is that $[\alpha^{-1},\beta_i)_{2,K}$ and $[\alpha^{-1},\beta_j)_{2,K}$ are inseparably linked if and only if the 3-fold Pfister form $\langle \! \langle \beta_i,\beta_j,\alpha^{-1}]\!]_K$ is hyperbolic by \cite{ChapmanGilatVishne:2017}, and since $\langle \! \langle \beta_i,\beta_j,\alpha^{-1}]\!]_F$ is anisotropic, it remains anisotropic under scalar extension tor any odd degree extension $K/F$.
Thus, if $K$ is the odd closure of $F$, then the quaternion algebras $[\alpha^{-1},\beta_1)_{2,K}$, $[\alpha^{-1},\beta_2)_{2,K}$ and $[\alpha^{-1},\beta_3)_{2,K}$ are not inseparably linked.

\section*{Acknowledgements}

The author thanks Jean-Pierre Tignol for his comments on the manuscript.
\bibliographystyle{abbrv}

\end{document}